\documentclass{amsart}
\usepackage{amsmath,amssymb,amsthm,graphics}
\usepackage[all]{xy}   
\usepackage{subfig}    
\usepackage{hyperref}

\input{glyphtounicode} \pdfgentounicode=1  

\usepackage{bbm}
\newcommand{\R}{\mathbb{R}}
\newcommand{\N}{\mathbb{N}}
\newcommand{\Z}{\mathbb{Z}}
\renewcommand{\1}{\mathbbm{1}}

\usepackage[normalem]{ulem} 
\usepackage{color}

\definecolor{darkgreen}{rgb}{0,0.5,0}

\newcommand{\arxiv}[1]{\href{http://www.arxiv.org/abs/#1}{arxiv:#1}}

\newtheorem*{maintheorem}{Main Theorem}
\newtheorem{theorem}{Theorem}[section]
\newtheorem{proposition}[theorem]{Proposition}

\newtheorem{lemma}[theorem]{Lemma}
\newtheorem{corollary}[theorem]{Corollary}
\theoremstyle{definition}
\newtheorem{definition}[theorem]{Definition}

\newtheorem{note}[theorem]{Note}

\newcommand{\newword}[1]{\textbf{#1}}

\newcommand{\VG}{V\mathcal{G}}
\newcommand{\G}{\mathcal{G}}
\newcommand{\K}{\mathcal{K}}

\newcommand{\V}{\mathfrak{V}}
\renewcommand{\P}{\mathcal{P}}
\renewcommand{\K}{X_{\mathrm{Stein}}}
\newcommand{\Knew}{X_{\mathrm{poly}}}   
\newcommand{\dlink}{\mathrm{lk}{\downarrow}}

\DeclareMathOperator{\stope}{psim}
\DeclareMathOperator{\supp}{supp}
\DeclareMathOperator{\rank}{rank}

\begin{document}

\title{R\"{o}ver's Simple Group is of Type $F_\infty$}

\author{James Belk}
\address{Mathematics Program, Bard College, Annandale-on-Hudson, NY 12504, USA.}
\email{\href{mailto:belk@bard.edu}{belk@bard.edu}}
\urladdr{\url{http://math.bard.edu/belk}}

\author{Francesco Matucci}
\address{Centro de \'Algebra da Universidade de Lisboa,
Avenida Professor Gama Pinto 2,
1649-003 Lisbon, Portugal.}
\email{\href{mailto:francesco.matucci@campus.ul.pt}{francesco.matucci@campus.ul.pt}}

\keywords{Thompson's groups, Grigorchuk's group, finiteness properties, polysimplicial complex}
\subjclass[2010]{20F65, 20J05, 20E08}

\begin{abstract}
We prove that Claas R\"{o}ver's Thompson-Grigorchuk simple group $\VG$ has type $F_\infty$.  The proof involves constructing two complexes on which $\VG$ acts: a simplicial complex analogous to the Stein complex for~$V$, and a polysimiplical complex analogous to the Farley complex for~$V$. We then analyze the descending links of the polysimplicial complex, using a theorem of Belk and Forrest to prove increasing connectivity.
\end{abstract}

\maketitle

\section{Introduction}
Let $\VG$ be the group of homeomorphisms of the Cantor set generated by Thompson's group $V$ and Grigorchuk's first group~$\G$.  This group was considered by Claas R\"{o}ver, who proved that $\VG$ is finitely presented and simple \cite{Rover1}, and also that $\VG$ is isomorphic to the abstract commensurator of~$\G$~\cite{Rover2}.

Recall that a group $G$ has \newword{type $\boldsymbol{F_\infty}$} if there exists a classifying space for $G$ with finitely many cells in each dimension.  The three Thompson groups $F$, $T$, and $V$ have type $F_\infty$ \cite{Brown1, BrownGeoghegan}, as do many of their variants such as the generalized groups $F_{n,k}$, $T_{n,k}$ and $V_{n,k}$~\cite{Brown1}, certain diagram groups~\cite{Farley1} and picture groups~\cite{Farley3}, braided Thompson groups~\cite{Bux}, higher-dimensional groups~$nV$~\cite{Fluch, KMN}, and various other generalizations \cite{BelkForrest, Cleary1, Cleary2, FarleyHughes, MMN, Matui}.

We prove the following theorem.

\begin{maintheorem}R\"{o}ver's group $\VG$ has type $F_\infty$.
\end{maintheorem}

Our basic approach is quite similar to that used in \cite{Bux} for the braided Thompson groups and in \cite{Fluch} for the higher-dimensional groups~$nV$.  Specifically, we begin by constructing a ranked poset $\mathcal{P}$ on which $\VG$ acts, and we show that the geometric realization $|\mathcal{P}|$ is contractible.  Next, we construct a contractible $\VG$-invariant subcomplex $\K$ of $|\mathcal{P}|$, which we refer to as the
Stein complex (see~\cite{Stein}), and we analyze the descending links of $\K$ with respect to the filtration induced by the rank.

Our methods for analyzing the descending links are new, and are simpler than those used in~\cite{Fluch}.  Specifically, we show that the Stein complex $\K$ is a simplicial subdivision of a certain complex $\Knew$ whose cells are products of simplices.  The descending links for this complex are flag complexes, and we use a simple combinatorial criterion (Theorem~\ref{thm:Grounded}) due to Belk and Forrest to prove that the connectivity of these flag complexes approaches infinity.

Nekrashevych \cite{Nek} has introduced a generalized family over R\"{o}ver-
type groups, obtained by combining a generalized Thompson's group $V_{n,1}$ with any self-similar group acting on an infinite rooted $n$-ary tree.
Unfortunately, our proof is very dependent on specific properties of the
Grigorchuk group, and does not generalize in an obvious way to the
Nekrashevych family of groups.  In particular, we use the fact that
Grigorchuk's group $\G$ is generated by a finite subgroup together with
certain elements of~$V$.  There are other self-similar groups with
analogous properties, e.g.~the Gupta-Sidki groups, and it should be
possible to modify our proof to work for these as well.

During the preparation of this manuscript, the authors became aware of some overlapping work by Geoghegan and Bartholdi~\cite{GeogBart}. Using somewhat different techniques, they prove that every R\"{o}ver-type group has type~$F_\infty$, provided that the underlying self-similar group is contracting.

\bigskip
\section{Notation and Background}
In this section we recall the necessary background material for Thompson's group~$V$, the first Grigorchuk group~$\mathcal{G}$, and R\"{o}ver's group~$\VG$.  We also extend~$V$ and $\VG$ to groupoids $\V$ and $\V\G$, respectively, which we will be using to define our complexes.  We present many results without proof, but in most cases the proofs can be found in either~\cite{PDLH} (for results on~$\G$), \cite{CFP} (for results on~$V$), or~\cite{Rover1} (for results on~$\VG$)

We will use the following notation.
\begin{itemize}
\item Throughout this paper functions are assumed to act on the left, with the product $fg$ denoting the composition $(fg)(x) = f(g(x))$.\smallskip
\item For each $n\in\N$, let $C(n)$ denote the disjoint union of $n$ copies of the Cantor set.
These will be the objects of the groupoids $\V$ and~$\V\G$. The first of these objects $C(1)$ is the ``canonical'' Cantor set, on which both Thompson's group $V$ and Grigorchuk's group~$\G$ act by homeomorphisms.\smallskip
\item If $f\colon C(m) \to C(m')$ and $g\colon C(n) \to C(n')$ are homeomorphisms, their \newword{direct sum} is the homeomorphism
\[
f\oplus g\colon C(m+n) \to C(m'+n')
\]
which maps the first $m$ Cantor sets of the domain to the first $m'$ Cantor sets of the range via~$f$, and maps the remaining $n$ domain Cantor sets to the remaining $n'$ range Cantor sets via~$g$.\smallskip
\item If $\alpha\in S_n$ is a permutation, the corresponding \newword{permutation homeomorphism} $p_\alpha\colon C(n)\to C(n)$ is the homeomorphism that permutes the Cantor sets of $C(n)$ according to~$\alpha$.\smallskip
\item Let $x\colon C(1) \to C(2)$ denote the \newword{split homeomorphism}, which maps the first half of $C(1)$ to the first Cantor set of~$C(2)$, and maps the second half of $C(1)$ to the second Cantor set of~$C(2)$.
\end{itemize}
Note that conjugating a direct sum of homeomorphisms of $C(1)$ by a permutation homeomorphism permutes the components of the sum,~i.e.
\[
(f_1\oplus\cdots\oplus f_n)p_\alpha \;=\; p_\alpha(f_{\alpha(1)}\oplus\cdots\oplus f_{\alpha(n)})
\]
for any homeomorphisms $f_i\colon C(1) \to C(1)$ and any $\alpha\in S_n$.

\subsection{Thompson's group $V$ and the groupoid $\V$}
For $n\in\N$ and $i\in\{1,\ldots,n\}$, let $x_i^{(n)}\colon C(n) \to C(n+1)$ denote the $i$th \newword{split homeomorphism},~i.e.
\[
x_i^{(n)} \;=\; \mathrm{id}_{i-1} \oplus x \oplus \mathrm{id}_{n-i}
\]
where $\mathrm{id}_k$ denotes the identity map on $C(k)$.

\begin{note}We will usually omit the parenthesized superscripts on split homeomorphisms, e.g.~writing $x_3$ instead of $x_3^{(5)}$.  In this case, the domain and range of $x_3$ must be determined from context.
\end{note}

We will refer to any composition of split homeomorphisms as a \newword{binary forest}.  If $f\colon C(m)\to C(n)$ is a binary forest, the $m$ Cantor sets in the domain of $f$ are called \newword{roots}, and the $n$ Cantor sets in the range are called \newword{leaves}.  A binary forest whose domain is $C(1)$ (so there is only one root) is called a \newword{binary tree}.

We can use binary forests to expand permutations, as described in the following proposition.

\begin{proposition}\label{prop:PermutationSplit}
Let $\alpha\in S_m$ be a permutation, and let $f\colon C(m) \to C(n)$ be a binary forest.  Then there exists a binary forest $f'\colon C(m)\to C(n)$ and a permutation $\alpha'\in S_n$ so that $f\alpha = \alpha' f'$.\hfill\qedsymbol
\end{proposition}

Let $\V$ be the groupoid with objects $\{C(n) \mid n\in\N\}$ generated by all split homeomorphisms and all permutation homeomorphisms.  Then the group of elements of~$\V$ that map $C(1)$ to $C(1)$ is \newword{Thompson's group $V$}.

Geometrically, elements of $\V$ can be thought of as braided diagrams (see \cite{GuSa}) or equivalently as abstract strand diagrams~(see \cite{BeMa}).  The following proposition is well-known.

\begin{proposition}
\label{thm:groupoid-V-correct-form}
Every element of\/ $\V$ can be written as $f_2^{-1}p_\alpha f_1$, where $f_1$ and $f_2$ are binary forests and $\alpha$ is a permutation.  In particular, every element of\/~$V$ can be written as $t_2^{-1}p_\alpha t_1$, where $t_1$ and $t_2$ are binary trees.\hfill\qedsymbol
\end{proposition}

\subsection{Grigorchuk's Group}
\label{ssec:grigo-group}
Let $\sigma$, $b$, $c$, and $d$ be the homeomorphisms of $C(1)$ defined by the following equations
\[
\sigma  \;=\; x^{-1}p_{(1\;2)}x,\;\;\; b  \;=\; x^{-1}(\sigma \oplus c)x,\;\;\; c \;=\; x^{-1}(\sigma \oplus d)x,\;\;\; d \;=\; x^{-1}(\1 \oplus b)x,
\]
where $\1$ denotes the identity homeomorphism on $C(1)$.  Note that these equations define $b$, $c$ and $d$ uniquely through recursion.

The group $\G = \langle\sigma,b,c,d\rangle$ is known as the \newword{first Grigorchuk group}.  See \cite{PDLH} for a general introduction to~$\G$, including the following proposition.

\begin{proposition}The generators $\sigma$, $b$, $c$, and $d$ all have order two.  Moreover, the four element set
\[
K \;=\; \{\1,b,c,d\}
\]
is a subgroup of $\G$ isomorphic to the Klein four-group.
\hfill\qedsymbol
\end{proposition}

Now, if $g$ is any element of Grigorchuk's group, then either
\[
xg = (g_1 \oplus g_2)x   \qquad\text{or}\qquad xg = p_{(1\;2)}(g_1\oplus g_2)x
\]
for some $g_1,g_2\in G$.  More generally, we can expand $g$ along any binary tree, as described in the following proposition.

\begin{proposition}
\label{prop:GrigoSplit}
If $g\in\G$ and $t\colon C(1)\to C(n)$ is a binary tree, then
\[
tg \;=\; p_\alpha(g_1 \oplus\cdots\oplus g_n)t'
\]
for some binary tree $t'\colon C(1)\to C(n)$, some $g_1,\ldots,g_n\in\G$, and some permutation $\alpha\in S_n$.

More generally, if $g_1,\ldots,g_m\in \G$ and $f\colon C(m)\to C(n)$ is a binary forest, then
\[
f(g_1\oplus\cdots\oplus g_n) \;=\; p_\alpha(g_1'\oplus\cdots\oplus g_n')f'
\]
for some binary forest $f'\colon C(m)\to C(n)$, some $g_1',\ldots,g_n'\in\G$, and some permutation $\alpha\in S_n$.\hfill\qedsymbol
\end{proposition}

The following proposition states that any element of $\G$ can be expanded to a particularly simple form.

\begin{proposition}\label{prop:GrigoForm}Let $g\in\G$.  Then there exist binary trees $t_1,t_2\colon C(1)\to C(n)$ so that
\[
g \;=\; t_2^{-1}p_\alpha(k_1\oplus\cdots\oplus k_n)t_1
\]
for some $\alpha\in S_n$ and $k_1,\ldots,k_n \in \{\1,b,c,d\}$.
\end{proposition}
\begin{proof}Recall that $\G$ is a contracting self-similar group with nucleus $\{\1,\sigma,b,c,d\}$ (see \cite{Nek}).  It follows that any $g\in\G$ can be written in the form
\[
g \;=\; t^{-1}p_\alpha(k_1\oplus\cdots\oplus k_n)t
\]
where $t\colon C(1) \to C(n)$ is a binary tree, $\alpha\in S_n$, and $k_1,\ldots,k_n\in\{\1,\sigma,b,c,d\}$.  If any of the $k_i$'s are equal to~$\sigma$, we can split the corresponding leaves to obtain the desired form.
\end{proof}

\begin{corollary}
\label{for:product-Grigo-correct-form}
Let $g_1,\ldots,g_m\in\G$.  Then there exists a pair of binary forests $f_1,f_2\colon C(m)\to C(n)$ so that
\[
g_1 \oplus \cdots \oplus g_m \;=\; f_2^{-1}p_\alpha (k_1 \oplus\cdots\oplus k_n)f_1
\]
for some $\alpha\in S_n$ and $k_1,\ldots,k_n \in \{\1,b,c,d\}$.\hfill\qedsymbol
\end{corollary}

\subsection{R\"{o}ver's group $\VG$ and the groupoid $\V\G$}
\newword{R\"{o}ver's group} $\VG$ is the group of homeomorphisms of $C(1)$ generated by the elements of $V$ and the elements of~$\G$.  More generally, \newword{R\"{o}ver's groupoid} $\V\G$ is the groupoid generated by the elements of $\V$ and the elements of~$\G$.  Roughly speaking, $\V\G$~is the groupoid consisting of all homeomorphisms $C(m)\to C(n)$ that locally look like elements of~$\VG$.

\begin{proposition}\label{prop:RoverForm}Every element of R\"{o}ver's groupoid $\V\G$ has the form
\[
f_2^{-1} p_\alpha (k_1 \oplus \cdots \oplus k_n) f_1
\]
where $f_1$ and $f_2$ are binary forests, $\alpha\in S_n$, and $k_1,\ldots,k_n\in\{\1,b,c,d\}$.

In particular, every element of R\"{o}ver's group $\VG$ has the form
\[
t_2^{-1} p_\alpha (k_1 \oplus \cdots \oplus k_n) t_1
\]
where $t_1,t_2\colon C(1)\to C(n)$ are binary trees, $\alpha\in S_n$, and $k_1,\ldots,k_n\in\{\1,b,c,d\}$.
\end{proposition}

\begin{proof}
By Proposition~\ref{thm:groupoid-V-correct-form}, elements of $\V$ have the required form.
Similarly, by Proposition~\ref{prop:GrigoForm}, every element of $\G$ also has the required form. Hence, to complete the proof
we just show that the products of two elements in $\V\G$ in the required form still
has the correct shape. Consider then a product of the form
\[
f_2^{-1}p_\alpha(k_1 \oplus \ldots \oplus k_n) f_1\;	h_2^{-1}
p_\beta(\ell_1 \oplus \ldots \oplus \ell_m) h_1
\]
Since $f_1h_2^{-1}p_\beta\in\V$, by Proposition \ref{thm:groupoid-V-correct-form} there exist binary forests $\widetilde{f}_1$ and $\widetilde{h}_2$ and a permutation $\gamma$ so that $f_1h_2^{-1}p_\beta = \widetilde{h}_2^{-1} p_\gamma \widetilde{f}_1$.  Then the above product can be written
\[
f_2^{-1}p_\alpha(k_1 \oplus \ldots \oplus k_n)
\widetilde{h}_2^{-1} p_\gamma
\widetilde{f}_1 (\ell_1 \oplus \ldots \oplus \ell_m) h_1
\]
Next, by Proposition~\ref{prop:GrigoSplit}, we know that $\widetilde{f}_1(\ell_1\oplus\cdots\oplus \ell_m) = p_\delta (g_1\oplus\cdots\oplus g_r)f_1'$
for some binary forest~$f_1'$, some permutation~$p_\delta$, and some $g_1,\ldots,g_r\in \G$, so the above product can be written
\[
f_2^{-1}p_\alpha(k_1 \oplus \ldots \oplus k_n)
\widetilde{h}_2^{-1} p_\epsilon
(g_1 \oplus \ldots \oplus g_m) h_1'
\]
where $h_1' = f_1'h_1$ and $p_\epsilon= p_\gamma p_\delta$.  Repeating the same step on the left and applying Proposition~\ref{prop:PermutationSplit}, we can rewrite this product as
\[
F_1^{-1}p_\zeta(g_1' \oplus \ldots \oplus g_m') p_\epsilon
(g_1 \oplus \ldots \oplus g_m) h_1'
\]
where $F_1$ is a binary forest, $p_\zeta$~is a permutation, and $g_1',\ldots,g_m'\in \G$.  Moving the $p_\epsilon$ to the left and combining the direct sums gives the form
\[
F_1^{-1}p_\eta(g_1'' \oplus \ldots \oplus g_m'') h_1'
\]
where $p_\eta = p_\zeta p_\epsilon$ and $g_i'' = g_{\gamma(i)}'g_i$.  This almost has the correct form---the only trouble is that $g_1'',\ldots,g_m''$ are arbitrary elements of~$\G$.  However, by Proposition~\ref{for:product-Grigo-correct-form}, we know that
\[
(g_1'' \oplus \ldots \oplus g_m'') \;=\; F_2^{-1}p_\theta(k_1' \oplus\cdots\oplus k_r')F_3
\]
for some permutation $\theta$ and some binary forests $F_2$ and $F_3$.  Then the original product can be written
\[
F_1^{-1}p_\epsilon F_2^{-1} p_\theta(k_1' \oplus\cdots\oplus k_r')F_3 h_1'
\]
Moving the $F_2^{-1}$ to the left and combining like terms
gives an expression in the desired form.
\end{proof}

\bigskip
\section{The Poset of Expansions}

In this section we define a poset $\P$ on which R\"{o}ver's group $\VG$ acts, and we show that the resulting geometric realization $|\P|$ is contractible.

For each $n$ and each $i\in\{1,\ldots,n\}$, let $\sigma_i^{(n)}$, $b_i^{(n)}$, $c_i^{(n)}$, and $d_i^{(n)}$ denote the homeomorphisms that act as $\sigma$, $b$, $c$, or $d$, respectively, on the $i$th Cantor set, and act as the identity elsewhere.  As with the split homeomorphism $x_i^{(n)}$, we will usually drop the parenthesized superscripts for these maps (writing only $\sigma_i$, $b_i$, $c_i$, or~$d_i$), in which case the domain must be determined from context.

Recall that $\G$ has a subgroup $K = \{\1,b,c,d\}$ isomorphic to the Klein four-group.  For each $n$, let $K_n$ denote the natural copy of the wreath product $K \wr S_n$ acting on~$C(n)$.  That is, let
\[
K_n \;=\; \bigl\{ p_\alpha(k_1\oplus\cdots\oplus k_n) \;\bigl|\; \alpha\in S_n\text{ and }k_1,\ldots,k_n\in K\bigr\}.
\]
If $f\colon C(1) \to C(n)$ is an element of R\"{o}ver's groupoid, let
\[
[f] \;=\; K_n f \;=\; \{kf \mid k\in K_n\},
\]
and let $\P$ be the set of all such cosets.  We shall refer to elements of $\P$ as \newword{vertices}, with the \newword{rank} of a vertex $[f]$ being the number of Cantor sets in the range of~$f$.

\begin{definition}Let $v,w\in\P$.  We say that $w$ is a \newword{splitting} of $v$ if there exists a homeomorphism $f\colon C(1)\to C(n)$ in~$\V\G$ and an $i\in\{1,\ldots,n\}$ so that
\[
v \;=\; [f]
\qquad\text{and}\qquad
w \;=\; [x_if].
\]
We say that $w$ is an \newword{expansion} of $v$, denoted $v \leq w$, if there exists a sequence of vertices $u_1,\ldots,u_m\in\P$ such that $u_1=v$, $u_m = w$, and each $u_{i+1}$ is a splitting of~$u_i$.
\end{definition}

Note that $\P$ forms a ranked poset under the expansion relation $\leq$.

\begin{proposition}\label{prop:DirectedSet}The poset $\P$ is a directed set.  That is, any two vertices in $\P$ have a common expansion.
\end{proposition}
\begin{proof}Let $[g]$ be a vertex in $\P$.  By Proposition~\ref{prop:RoverForm}, we know that
\[
g \;=\; f^{-1} p_\alpha (k_1\oplus\cdots\oplus k_n) t
\]
for some binary forest $f$, some binary tree $t$, some permutation $\alpha\in S_n$, and some elements $k_1,\ldots,k_n\in K$.  Since $f$ is a composition of split homeomorphisms, the vertex $[fg]$ is an expansion of $[g]$.  But
\[
[fg] \;=\; [p_\alpha (k_1\oplus\cdots\oplus k_n) t] \;=\; [t]
\]
since $p_\alpha(k_1\oplus\cdots\oplus k_n)\in K$.  Thus every vertex in $\P$ has an expansion which is just a binary tree.  But clearly any two binary trees have a common expansion.
\end{proof}

Let $|\P|$ denote the geometric realization of the poset $\P$, i.e.~the simplicial complex whose vertices are elements of~$\P$, with simplices corresponding to finite chains $v_1 < \cdots < v_k$.

\begin{corollary}\label{cor:ContractibleP}The geometric realization $|\P|$ of $\P$ is contractible.
\end{corollary}
\begin{proof}It is well known that the geometric realization of any directed set is contractible. See \cite[Prop.~9.3.14]{Geoghegan} for a proof.
\end{proof}

Note that R\"{o}ver's group $\VG$ acts on the vertex set $\P$ on the right by pre-composition, i.e.~$[f]g=[fg]$ for all $f\colon C(1)\to C(n)$ in~$\V\G$ and $g\in\VG$.  It follows that $\VG$ acts simplicially on~$|\P|$.

\begin{proposition}\label{prop:FiniteStabilizers}Under the action of $\VG$, each vertex in $|\P|$ has finite stabilizer.
\end{proposition}
\begin{proof}If $[f]$ is any vertex of rank~$n$, then the stabilizer of $[f]$ is precisely the group $f^{-1}K_n f$. This is isomorphic to the wreath product $K\wr S_n$, which is finite of order~$n!\cdot 4^n$.
\end{proof}

Unfortunately, the complex $|\P|$ is too large for us to successfully apply Brown's criterion.  As with other Thompson-like groups, it will be necessary to consider a certain subcomplex of $|\P|$, which we will define in the next section.

\bigskip
\section{The Stein Complex}

In this section we define a locally finite $\VG$-invariant subcomplex $\K$ of $|\P|$, and we prove that $\K$ is contractible.  The complex $\K$ is the analog of the complexes for $F$, $T$, and $V$ introduced by Stein in~\cite{Stein}. Similar complexes were introduced in \cite{Bux} and \cite{Fluch} for the braided Thompson groups $BV$ and the higher-dimensional Thompson groups~$sV$, respectively.

Before defining~$\K$, we need some more information about splittings. Although we have defined splittings using the split homeomorphisms~$x_i^{(n)}$, the form of a splitting may depend on a chosen representative~$f$.
For example, if $[f]$ is a vertex of rank $n$, then $b_i f$ is also a representative for~$f$.  But
\[
x_i^{(n)} b_i^{(n)} \;=\; \sigma_{i}^{(n+1)} c_{i+1}^{(n+1)} x_i^{(n)}
\]
so
\[
[x_i b_i f] \;=\; [\sigma_i c_{i+1} x_i f] \;=\; [\sigma_i x_i f]
\]
where the last equality follows from the fact that $c_{i+1} \in K_{n+1}$.  We conclude that $[\sigma_i x_i f]$ is a splitting of~$[f]$.

The following proposition shows that these are the only ``unusual'' splittings.

\begin{proposition}\label{prop:ClassifySplittings}
Let $[f]$ be a vertex in $\P$ of rank $n$.  Then every splitting of $[f]$ has the form
\[
[x_if]\qquad\text{or}\qquad [\sigma_i x_i f]
\]
for some $i\in\{1,\ldots,n\}$
\end{proposition}
\begin{proof}Let $g$ be any other representative for $[f]$.  Then $g\in K_n f$, so
\[
g \;=\; p_\alpha (k_1\oplus\cdots\oplus k_n) f
\]
for some $\alpha\in S_n$ and $k_1,\ldots,k_n \in K$.  If $i\in\{1,\ldots,n\}$, then
\[
x_ig \;=\; x_i p_\alpha (k_1\oplus\cdots\oplus k_n) f.
\]
But $x_i p_\alpha = p_\beta x_j$ for some $\beta\in S_{n+1}$, where $j = \alpha^{-1}(i)$.  So
\[
[x_ig] \;=\; [p_\beta x_j (k_1\oplus\cdots\oplus k_n) f] \;=\; [x_j (k_1\oplus \cdots \oplus k_n)f].
\]
Since $k_j \in \{\1,b,c,d\}$, we conclude that $[x_ig]$ is either $[x_j f]$, $[x_j b_j f]$, $[x_j c_j f]$, or $[x_j d_j f]$.  We already know that $[x_j b_j f] = [\sigma_j x_j f]$, and similarly
\[
[x_j c_j f] \;=\; [\sigma_j d_{j+1} x_j f] \;=\; [\sigma_j x_j f]
\qquad\text{and}\qquad
[x_j d_j f] \;=\; [b_{j+1} x_j f] \;=\; [x_j f].\qedhere
\]
\end{proof}

Thus, there are exactly two ways to split the $i$th Cantor set of~$f$: we can compose with either $x_i$ or $\sigma_i x_i$.  By Proposition~\ref{prop:DirectedSet}, these two splittings $[x_if]$ and $[\sigma_i x_i f]$ must have a common expansion.  Indeed, $[x_ix_if]$ is an expansion of them both, since
\[
[x_i\sigma_i x_i f] \;=\; [p_{(i\;i+1)} x_i x_i f] \;=\; [x_i x_i f].
\]
To clarify the situation further, the following picture shows a portion of $\P$ lying above the vertex $[f]$.
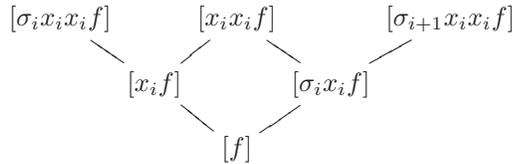
\[
    \xymatrix@R=2ex@C=0em{
    [\sigma_i x_i x_i f]\ar@{-}[dr] & & [x_i x_i f]\ar@{-}[dl]\ar@{-}[dr] & & [\sigma_{i+1} x_i x_i f]\ar@{-}[dl] \\
    & [x_i f]\ar@{-}[dr] & & [\sigma_i x_i f]\ar@{-}[dl] & \\
    & & [f] & &
    }
\]
This prompts the following definition.

\begin{definition}Let $v,w\in\P$.  We say that $w$ is a \newword{double splitting} of $v$ if there exists an $f\colon C(1)\to C(n)$ in $\V\G$ so that
\[
v \;=\; [f]
\qquad\text{and}\qquad
w \;=\; [x_i x_i f]
\]
for some $i\in\{1,\ldots,n\}$.
\end{definition}

As the following proposition shows, double splittings do not have the same ambiguity as single splittings.

\begin{proposition}
\label{thm:double-splitting}
Let $[f]$ be a vertex in $\P$ of rank $n$.  Then every double splitting of $[f]$ has the form $[x_ix_i f]$ for some $i\in\{1,\ldots,n\}$.
\end{proposition}
\begin{proof}Following the proof of Proposition~\ref{prop:ClassifySplittings}, we find that the only possible double splittings of $[f]$ are $[x_ix_if]$, $[x_ix_ib_if]$, $[x_ix_i c_if]$, and $[x_ix_i d_if]$.  But
\[
[x_ix_i b_i f] \;=\; [x_i \sigma_i c_{i+1} x_i f] \;=\; [p_{(i\;i+1)} c_{i+2} x_ix_i f] \;=\; [x_ix_if]
\]
and
\[
[x_ix_i c_i f] \;=\; [x_i \sigma_i d_{i+1} x_i f] \;=\; [p_{(i\;i+1)} d_{i+2} x_ix_i f] \;=\; [x_ix_if]
\]
and
\[
[x_ix_i d_i f] \;=\; [x_i b_{i+1} x_i f] \;=\; [b_{i+2} x_ix_i f] \;=\; [x_ix_if].\qedhere
\]
\end{proof}

We are now ready to define the complex $\K$.

\begin{definition}If $[f]$ is a vertex in $\P$ of rank $n$, an \newword{elementary expansion} of $[f]$ is any vertex of the form
\[
[(u_1 \oplus \cdots \oplus u_n)f]
\]
where each $u_i \in \{\1,x,\sigma_1x,x_1x\}$.
\end{definition}

That is, an elementary expansion of $[f]$ is obtained by splitting or double splitting some of the Cantor sets in the range of~$f$.  Note that this definition does not depend on the chosen representative~$f$.

\begin{definition}A simplex $v_1 < \cdots < v_n$ in $|\P|$ is called an \newword{elementary simplex} if $v_n$ is an elementary expansion of $v_1$.  The \newword{Stein complex} $\K$ for $\VG$ is the subcomplex of $|\P|$ consisting of all elementary simplices.
\end{definition}

We wish to prove that $\K$ is contractible.  To do so, consider the \newword{intervals} in~$\P$, which are subsets of the form
\[
[u,w] \;=\; \{v\in\P \mid u\leq v \leq w\}.
\]
We wish to prove that every nonempty interval $[u,w]$ in $\P$ contains a maximum elementary expansion of~$u$, which we refer to as the \newword{elementary core} of the interval.  That is, $v_0 \in [u,w]$ is the elementary core of $[u,w]$ if $v_0$ is an elementary expansion of~$u$, and $v_0$ is a common expansion of all elementary expansions of $u$ contained in $[u,w]$.

\begin{lemma}\label{lemma:ClassifyingExpansions}Let $v\in\P$ be a vertex of the form\/ $[(g_1 \oplus\cdots\oplus g_n)f]$, where $g_1,\ldots,g_n$ and $f$ are homeomorphisms in~$\V\G$.  Then the expansions of $v$ are precisely the vertices of the form\/ $[(h_1\oplus\cdots\oplus h_n)f]$, where each\/ $[h_i]$ is an expansion of\/ $[g_i]$.
\end{lemma}
\begin{proof}Note first that $v$ itself has the required form, with $h_i = g_i$ for each~$i$.  By Proposition~\ref{prop:ClassifySplittings}, each subsequent splitting is just a composition by $x_i$ or $\sigma_ix_i$, and is therefore equivalent to a splitting of one of the $h_j$'s.
\end{proof}

\begin{lemma}\label{lemma:ElemCoreIdentity}Let $\1$ denote the identity map on $C(1)$, let $u = [\1]$, and let $w$ be an expansion of~$u$. Then the interval $[u,w]$ has an elementary core.
\end{lemma}
\begin{proof}There are only four elementary expansions of $[\1]$, as shown in the following picture.
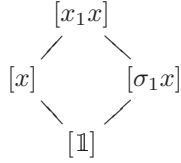
\[
    \xymatrix@R=2ex@C=0em{
    & [x_1x]\ar@{-}[dr]\ar@{-}[dl] & \\
    [x]\ar@{-}[dr] & & [\sigma_1x]\ar@{-}[dl] & \\
    & [\1] &
    }
\]
Thus it suffices to prove that $[x_1x] \in [u,w]$ whenever $[x] \in [u,w]$ and $[\sigma_1x] \in [u,w]$.

Suppose that $[x] \in [u,w]$ and $[\sigma_1 x]\in [u,w]$, so $[x] \leq w$ and $[\sigma_1 x] \leq w$.  Note that $[\1]$ has only six expansions of rank two or less:
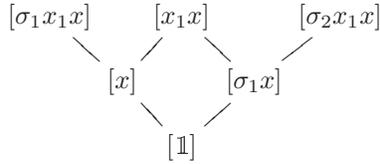
\[
    \xymatrix@R=2ex@C=0em{
    [\sigma_1 x_1 x]\ar@{-}[dr] & & [x_1 x]\ar@{-}[dl]\ar@{-}[dr] & & [\sigma_2 x_1 x]\ar@{-}[dl] \\
    & [x]\ar@{-}[dr] & & [\sigma_1 x]\ar@{-}[dl] & \\
    & & [\1] & &
    }
\]
If $w$ is an expansion of $[x_1x]$ then we are done, so suppose instead that $w$ is a common expansion of $[\sigma_1 x_1 x]$ and $[\sigma_2 x_1 x]$.
Note that
\[
\sigma_1 x_1 x \;=\; (\sigma \oplus \1 \oplus \1) x_1x
\qquad\text{and}\qquad
\sigma_2 x_1 x \;=\; (\1 \oplus \sigma \oplus \1) x_1x
\]
Since $[\sigma_1 x_1 x] \leq w$, Lemma~\ref{lemma:ClassifyingExpansions} tells us that
\[
w \;=\; (f_1 \oplus f_2 \oplus f_3)x_1x
\]
where
\[
[\sigma]\leq [f_1],
\qquad
[\1] \leq [f_2],
\qquad\text{and}\qquad
[\1] \leq [f_3].
\]
But since $[\sigma_2 x_1 x] \leq w$, we also know that
\[
[\1]\leq [f_1],
\qquad
[\sigma] \leq [f_2],
\qquad\text{and}\qquad
[\1] \leq [f_3].
\]
Then $[f_1]$, $[f_2]$, and $[f_3]$ are all expansions of $[\1]$, and therefore $w$ is an expansion of $[(\1\oplus\1\oplus \1)x_1x]=[x_1x]$ by Lemma~\ref{lemma:ClassifyingExpansions}.
\end{proof}

\begin{proposition}Every nonempty interval\/ $[u,w]$ in $\P$ has an elementary core.
\end{proposition}
\begin{proof}Let $f \in \V\G$ so that $u=[f]$.  Since $w$ is an expansion of $u$, we know that
\[
w \;=\; [(g_1 \oplus\cdots \oplus g_n)f]
\]
for some expansions $[g_1],\ldots,[g_n]$ of $[\1]$, where $n$ is the rank of $[f]$.  For each $i$, the interval $\bigl[[\1],[g_i]\bigr]$ has an elementary core $[h_i]$ by Lemma~\ref{lemma:ElemCoreIdentity}, where each $h_i$ is in $\{\1,x,\sigma_1x,x_1x\}$.  We claim that $v_0 = [(h_1\oplus\cdots\oplus h_n)f]$ is an elementary core for~$[u,w]$.

First note that $v_0$ is an elementary expansion of~$u$.  Now let $v$ be any elementary expansion of $u$ such that $v\in[u,w]$.  We know that
\[
v \;=\; [(h_1' \oplus\cdots \oplus h_n')f]
\]
for some $h_1',\ldots,h_n'\in\{\1,x,\sigma_1x,x_1x\}$.  Since $v \leq w$, we also know that $[h_i'] \leq [g_i]$ for each $i$.  Then $[h_i']$ is an elementary expansion of $\1$ and $[h_i'] \in \bigl[[\1],[g_i]\bigr]$, so $[h_i'] \leq [h_i]$ for each~$i$. By Lemma~\ref{lemma:ClassifyingExpansions}, it follows that $v\leq v_0$.
\end{proof}

Note that the elementary core of $[u,w]$ is only equal to $u$ in the case where $u=w$.  For the following proof, we need a proposition of Quillen's.

\begin{proposition}\label{prop:Quillen}Let $X$ be a poset, and suppose there exists an element $x_0\in X$ and a function $f\colon X\to X$ so that
\[
x \geq f(x) \leq x_0
\]
for all $x\in X$.  Then the geometric realization $|X|$ is contractible.
\end{proposition}
\begin{proof}See \cite{Quillen}, Section~1.5.
\end{proof}

We say that an interval $[v,w]$ in $\P$ is \newword{non-elementary} if $v \leq w$ and $w$ is not an elementary expansion of $v$.

\begin{lemma}\label{lemma:ContractibleInterval}Let $[u,w]$ be a non-elementary interval in $\P$, and let
\[
(u,w) \;=\; \{v\in \P \mid u<v<w\}.
\]
Then the geometric realization $|(u,w)|$ is contractible.
\end{lemma}

This proof is the same as the proof in Lemma~2.4 of \cite{Fluch}, which itself derives from the proof of the lemma in Section~4 of~\cite{Brown2}.

\begin{proof}Let $v_0$ be the elementary core of $[u,w]$, and note that $v_0 \in (u,w)$ since $w$ is not an elementary expansion of~$u$.  For each $v \in (u,w)$, let $f(v)$ be the elementary core of the interval $[u,v]$, and note that $f(v)$ is always an element of~$(u,w)$.  Moreover, $f(v) \leq v$ and $f(v) \leq v_0$ for all $v\in (u,w)$, and therefore $|(u,w)|$ is contractible by Proposition~\ref{prop:Quillen}.
\end{proof}

\begin{proposition}The complex $\K$ is contractible
\end{proposition}

Again, this proof is the same as the proof in Corollary~2.5 of \cite{Fluch}, which itself derives from a proof in~\cite{Brown2}.

\begin{proof}Define the \textit{length} of a non-elementary interval $[v,w]$ in $\mathcal{P}$ to be the difference of the ranks of $v$ and~$w$.  Suppose we start with~$\K$, and attach the geometric realizations $|[v,w]|$ of non-elementary intervals in $\P$ increasing order of length.  Clearly each $|[v,w]|$ is contractible.  Moreover, each $|[v,w]|$ is being attached along $|[v,w)| \cup |(v,w]|$, which is simply the suspension of $|(v,w)|$.  Since $|(v,w)|$ is contractible by Lemma~\ref{lemma:ContractibleInterval}, it follows that $|[v,w)| \cup |(v,w]|$ is contractible, and therefore attaching $|[v,w]|$ does not change the homotopy type.  But the end result of attaching all of these complexes is~$|\P|$, which is contractible by Corollary~\ref{cor:ContractibleP}, so $\K$ must itself be contractible
\end{proof}

\bigskip
\section{A Polysimplicial Complex}
In this section, we introduce a polysimplicial complex~$\Knew$ of which $\K$ is a simplicial subdivision.  Here the word \newword{polysimplex} refers to any Euclidean polytope obtained by taking a product of simplices.  Thus a \newword{polysimplicial complex} is an affine cell complex whose cells are polysimplices, with the property that the intersection of any two non-disjoint cells is a common face of each.

\begin{note}
This notion of a polysimplicial complex is more general than the one introduced by Bruhat and Tits in \cite{BrTi} and used in the theory of buildings.  In particular, we place no requirements on the dimensions of the polysimplices, and we do not require the existence of galleries joining pairs of cells.
\end{note}

Polysimplicial complexes are a common generalization of simplicial complexes and cubical complexes.  Note that cubes are indeed polysimplicial, being products of $1$-simplices.  The polysimplicial complex $\Knew$ that we will define can be viewed as an analogue for $\VG$ of Farley's cubical complexes for $F$, $T$, and $V$ (see \cite{Farley1,Farley2}).  The Stein complexes for the Brin-Thompson groups $nV$ defined in \cite{Fluch} are also simplicial subdivisions of polysimplicial complexes, and the approach we use here to analyze the descending links of $\Knew$ would work just as well for these complexes.

We begin by defining a collection of simplicially subdivided polysimplices within our complex $\K$.

\begin{definition}Let $[f]$ be a vertex of rank $n$ in~$\K$, and for each $i\in \{1,\ldots,n\}$ let $S_i$ be one of the following sets:
\[
\{\1\},\quad \{\1,x\},\quad\{\1,x,x_1 x\},\quad\{\1,\sigma_1 x\},\quad \{\1,\sigma_1 x, x_1 x\},\quad\text{or}\quad \{\1,x_1 x\}.
\]
Then the corresponding \newword{basic polysimplex} in $\K$, denoted $\stope(f,S_1,\ldots,S_n)$, is the full subcomplex of $\K$ spanned by the following set of vertices:
\[
\bigl\{[(u_1\oplus\cdots \oplus u_n)f] \;\bigl|\; u_i\in S_i\text{ for all }i\bigr\}.
\]
\end{definition}

\smallskip Each basic polysimplex has the combinatorial structure of a simplicial subdivision of a polysimplex.  In particular,
\[
\stope(f,S_1,\ldots,S_n) \;\cong\; \Delta^{d_1}\times \cdots \times \Delta^{d_n},
\]
where $\Delta^{d_i}$ denotes a simplex of dimension $d_i = |S_i| - 1$.

Note that if $f$ and $f'$ are two representatives for the same vertex, then every basic polysimplex $\stope(f',S_1',\ldots,S_n')$ can be written as $\stope(f,S_1,\ldots,S_n)$ for some sets $S_1,\ldots,S_n$.  That is, the basic polysimplices based at a vertex $[f]$ do not depend on the chosen representative~$f$.

\begin{lemma}
\label{thm:intersection-polysimplices}
The intersection of two non-disjoint basic polysimplices is a common face of each.
\end{lemma}
\begin{proof}Let $P = \stope(f,S_1,\ldots,S_m)$ and $Q = \stope(g,T_1,\ldots,T_n)$ be two basic polysimplices with nonempty intersection.  Define a binary operation $\land$ on the vertices of $P$ by the formula
\[
[(s_1 \oplus \cdots \oplus s_m)f] \land [(s_1' \oplus \cdots \oplus s_m')f] \;=\; \bigl[\bigl(\min(s_1,s_1'
) \oplus \cdots \oplus \min(s_m,s_m')\bigr) f\bigr].
\]
and define a similar binary operation on the vertices of $Q$.  We claim that the two definitions of $\land$ agree on the vertices of $P\cap Q$.

Let $v$ and $v'$ be vertices of $P\cap Q$.  Note that the definition of $\land$ is preserved by restrictions to faces.  Therefore, without loss of generality, we may assume that $v\land v' = [f]$ in~$P$, and $v\land v' = [g]$ in~$Q$. Then
\[
v \;=\; [(s_1 \oplus \cdots \oplus s_m)f] \;=\; [(t_1\oplus\cdots\oplus t_n)g]
\]
for some $s_i\in S_i$ and $t_i\in T_i$, and similarly
\[
v' \;=\; [(s_1' \oplus \cdots \oplus s_m')f] \;=\; [(t_1'\oplus\cdots\oplus t_n')g]
\]
for some $s_i'\in S_i$ and $t_i'\in T_i$. Then
\begin{align*}
(s_1 \oplus \cdots \oplus s_m)f =p_{\alpha}(k_1\oplus\cdots\oplus k_r) (t_1\oplus\cdots\oplus t_n)g \\[3pt]
\text{and}\qquad(s_1' \oplus \cdots \oplus s_m')f =p_{\beta}(k_1'\oplus\cdots\oplus k_p') (t_1'\oplus\cdots\oplus t_n')g
\end{align*}
for some $k_1,\ldots,k_r,k_1',\ldots,k_p'\in K$ and permutations $\alpha$ and $\beta$.  Solving for $fg^{-1}$ in both of these equations gives
\begin{multline*}
fg^{-1} \;=\; (s_1 \oplus \cdots \oplus s_m)^{-1}p_{\alpha}(k_1\oplus\cdots\oplus k_r) (t_1\oplus\cdots\oplus t_n) \\[3pt]
\;=\;
(s_1' \oplus \cdots \oplus s_m')^{-1}p_{\beta}(k_1'\oplus\cdots\oplus k_p') (t_1'\oplus\cdots\oplus t_n').
\end{multline*}
Now, since $v\land v' = [f]$, we know that $\min(s_i,s_i') = \1$
for each $i$, so either $s_i = \1$ or $s_i' = \1$ for each~$i$, and the same holds true for $t_i$ and~$t_i'$.
Then the only possibility is that $m=n$ and
\[
fg^{-1} \;=\; p_\gamma (k_1''\oplus\cdots\oplus k_n'')
\]
for some permutation $\gamma$ and some $k_1'',\ldots,k_n'' \in K$, and hence $[f]=[g]$.  This proves that the two definitions of $\land$ agree on $P\cap Q$.


Now, let $v$ be the vertex $v_1\land \cdots \land v_k$, where $v_1,\ldots,v_k$ are the vertices of~$P\cap Q$.  Then $v$ must be a vertex of $P\cap Q$, and indeed is a minimum for the vertices of $P\cap Q$. Note that the full subcomplex of $P$ spanned by the vertices of $P$ that are greater than or equal to~$v$ is a face of~$P$, and similarly the full subcomplex of $Q$ spanned by the vertices of $Q$ that are greater than or equal to $Q$ is a full subcomplex of~$Q$.  Therefore, without loss of generality, we may assume that $[f] = [g] = v$.  Indeed, we may as well assume that $f=g$.  Then
\[
P\cap Q \;=\; \stope(f,S_1\cap T_1,\ldots,S_m\cap T_m)
\]
which is a common face of each.
\end{proof}

\begin{proposition}The basic polysimplices in the Stein complex $\K$ form a polysimplicial complex~$\Knew$, which has $\K$ as a simplicial subdivision.
\end{proposition}
\begin{proof}Note first that each simplex of $\K$ lies in the interior of a unique basic polysimplex. Specifically, given a $k$-simplex $\Delta = (v_0 < \cdots < v_k)$ in $\K$, let $f$ be a representative for~$v_0$. Then each vertex $v_i$ of this simplex has the form
\[
v_i \;=\; [(u_{i,1} \oplus \cdots \oplus u_{i,n})f]
\]
for some $u_{i,j} \in \{\1,x,\sigma_1x,x_1x\}$, so $\Delta$ is contained in the interior of the basic polysimplex $\stope(f,S_1,\ldots,S_n)$, where each $S_j = \{u_{0,j},u_{1,j},\ldots,u_{k,j}\}$.

It should be clear from the definition that each face of a basic polysimplex is again a basic polysimplex.  Furthermore, Lemma \ref{thm:intersection-polysimplices}
shows that the intersection of two non-disjoint basic polysimplices is a common face of each.  We conclude that $\Knew$ is a polysimplicial complex.
\end{proof}

Note that the vertices of $\Knew$ are all the elements of $\P$ (i.e.~the same vertices as~$\K$) and each edge of $\Knew$ corresponds to either a splitting or a double splitting of a vertex in~$\P$. Note also that elements of $\VG$ map basic polysimplices to basic polysimplices, and therefore $\VG$ acts on the complex $\Knew$.

By the way, even though $\K$ is a simplicial subdivision of $\Knew$, it is \textit{not} simply the barycentric subdivision of~$\Knew$. For example, each square of $\Knew$ is the union of two triangles from~$\K$.

Because a polysimplicial complex is an affine cell complex, we can apply Bestvina-Brady Morse theory~\cite{BestvinaBrady} to~$\Knew$. This is based on the following definition.

\begin{definition}Let $X$ be an affine cell complex.  A \newword{Morse function} on $X$ is a map $\phi\colon X\to\R$ such that
\begin{enumerate}
\item $\phi$ restricts to a non-constant affine linear map on each cell of $X$ of dimension one or greater, and\smallskip
\item the image under $\phi$ of the $0$-skeleton of $X$ is discrete in $\R$.
\end{enumerate}
If $\phi$ is a Morse function on $X$ and $r\in\R$, the \newword{sublevel complex} $X^{\leq r}$ is the subcomplex of $X$ consisting of all cells that are contained in $\phi^{-1}\bigl((-\infty,r]\bigr)$.  If $v$ is a vertex in $X$, the \newword{descending link} of $v$ is its link in the corresponding sublevel complex:
\[
\dlink(v) \,=\, \mathrm{lk}\bigl(v,X^{\leq\phi(v)}\bigr).
\]
\end{definition}

Note that, if $X$ is a polysimplicial complex, then the descending link of any vertex $v$ in $X$ is a simplicial complex.  If $X$ itself is not simplicial, this descending link cannot be viewed as a subcomplex of~$X$.  For example, although each vertex of $\dlink(v)$ corresponds to a vertex of $X^{\leq\phi(v)}$ that is adjacent to~$v$, two such vertices are connected by an edge in $\dlink(v)$ if and only if the corresponding vertices of $X^{\leq\phi(v)}$ lie in a common $2$-cell containing~$v$.

By now, the following combination of the Betvina-Brady Morse lemma \cite{BestvinaBrady} with Brown's criterion \cite{Brown1} is standard.

\begin{theorem}Let $G$ be a group acting cellularly on a contractible affine cell complex~$X$, and let $\phi\colon X\to \R$ be a Morse function on~$X$. Suppose that:
\begin{enumerate}
\item Each sublevel complex $X^{\leq r}$ has finitely many orbits of cells.\smallskip
\item The stabilizer of each vertex in $X$ is finite.\smallskip
\item For each $k\in\N$, there exists an $r\in\R$ so that the descending link of each vertex in $\phi^{-1}\bigl([r,\infty)\bigr)$ is $k$-connected.
\end{enumerate}
Then $G$ has type $F_\infty$.
\end{theorem}

Now, define a Morse function $\phi$ on our polysimplicial complex $\Knew$ by defining the value of $\phi$ on each vertex to be its rank in the poset~$\P$, and then extending linearly to each polysimplex.  Since the endpoints of each edge in $\Knew$ have different ranks, $\phi$~is non-constant on each polysimplex of dimension one or greater, and thus $\phi$ is a valid Morse function.

To prove that $\VG$ has type $F_\infty$, we must prove that $\Knew$ satisfies conditions (1) through (3) of the above theorem.  We begin with condition~(1).

\begin{proposition}Each sublevel complex $\Knew^{\,\leq r}$ has finitely many $\VG$-orbits of cells.
\end{proposition}
\begin{proof}Note that any two vertices $[f],[g]\in\P$ of the same rank are in the same $\VG$-orbit, since $f^{-1}g\in\VG$ and $f^{-1}g$ maps $[f]$ to~$[g]$.  Therefore, each sublevel complex has only finitely many orbits of vertices.  More generally, observe that $f^{-1}g$ maps the cell $\stope(f,S_1,\ldots,S_n)$ to the cell $\stope(g,S_1,\ldots,S_n)$, and therefore each sublevel complex has only finitely many orbits of cells.
\end{proof}

This verifies condition (1), and condition (2) is the content of Proposition~\ref{prop:FiniteStabilizers}.  Therefore, all that remains is to show condition~(3) on the connectivity of the descending links.  Specifically, we must show that for each $k\in\N$, there exists an $n\in\N$ so that the descending link of each vertex in $\Knew$ of rank $n$ or greater is $k$-connected.  The proof of this condition is given in the next section.

\bigskip
\section{Descending Links}

In this section, we complete the proof that $\VG$ has type $F_\infty$ by analyzing the descending links of the polysimplicial complex $\Knew$.  Our approach is based on the following definition and theorem, which are due to the first author and Bradley Forrest~\cite{BelkForrest}, and have not previously appeared in published form.

\begin{definition}[Belk, Forrest]Let $X$ be a simplicial complex, and let $k\geq 1$.
\begin{enumerate}
\item A simplex $\Delta$ in $X$ is called a \newword{$\boldsymbol{k}$-ground} for $X$ if every vertex of $X$ is adjacent to all but at most $k$ vertices of~$\Delta$.\smallskip

\item We say that $X$ is \newword{$\boldsymbol(n,k)$-grounded} if there exists an $n$-simplex in $X$ that is a $k$-ground for~$X$.
\end{enumerate}
\end{definition}

Note that any sub-simplex of a $k$-ground for $X$ is again a $k$-ground for $X$.  Thus an $(n,k)$-grounded complex is also $(n',k)$-grounded for all $n' < n$.

For the following theorem, recall that a \newword{flag complex} is a simplicial complex $X$ with the property that every finite set of vertices that are pairwise joined by edges spans a simplex in~$X$.

\begin{theorem}[Belk, Forrest]\label{thm:Grounded}For $m,k\geq 1$, every finite $(mk,k)$-grounded flag complex is \mbox{$(m-1)$-connected}.
\end{theorem}
\begin{proof}We proceed by induction on $m$.  For $m=1$, the statement is that every finite $(k,k)$-grounded flag complex is connected, which is clear from the definition.

Now suppose that every finite $(mk,k)$-grounded flag complex is $(m-1)$-connected, and let $X$ be a finite $\bigl((m+1)k,k\bigr)$-grounded flag complex.  Then we can filter $X$ as a chain of full subcomplexes
\[
\Delta = X_0 \subset X_1 \subset \cdots \subset X_p = X.
\]
where $\Delta$ is an $(m+1)k$-simplex that is a $k$-ground for~$X$, and each $X_i$ is obtained from $X_{i-1}$ by adding a single vertex~$v_i$.

Let $L_i$ denote the link of $v_i$ in $X_i$, and observe that each $X_i$ is homeomorphic to the union $X_{i-1} \cup_{L_i} CL_i$, where $CL_i$ denotes the cone on~$L_i$.  Since $\Delta$ is a \mbox{$k$-ground} for~$X$, we know that $L_i$ includes at least $mk+1$ vertices of~$\Delta$.  In particular, the intersection $L_i \cap \Delta$ contains an \mbox{$mk$-simplex}, which must be a $k$-ground for~$L_i$. By our induction hypothesis, it follows that each $L_i$ is \mbox{$(m-1)$-connected}.  Since $X_0 = \Delta$ is contractible, this proves that $X_i$ is \mbox{$m$-connected} for every~$i$, and in particular $X$ is \mbox{$m$-connected}.
\end{proof}

Now consider the complex $\Knew$.  We wish to show that the connectivity of the descending links in $\Knew$ goes to infinity.  That is, we wish to show that for each $k\in\mathbb{N}$, there exists an $n\in\mathbb{N}$ so that for any vertex $v$ in $\Knew$ of rank $n$ or greater, the descending link $\dlink(v)$ is $k$-connected.

If $v\in\P$, a vertex $w\in\P$ is called a \newword{contraction} of $v$ if $v$ is either a splitting or a double splitting of~$w$.  Note that the contractions of $v$ are in one-to-one correspondence with the vertices of~$\dlink(v)$ in~$\Knew$. We will use the following notation for contractions:
\begin{itemize}
\item If $[f]$ is a vertex of rank $n$ and $i,j\in\{1,\ldots,n\}$ are distinct, let
\[
[C_{ij}f] \;=\; [x_1^{-1}p_\alpha f],
\]
where $\alpha\in S_n$ is any permutation for which $\alpha(i) = 1$ and $\alpha(j) = 2$.\smallskip
\item If $[f]$ is a vertex of rank $n$ and $i,j,k\in\{1,\ldots,n\}$ are distinct, let
\[
[C_{ijk}f] \;=\; [x_1^{-1}x_1^{-1} p_\alpha f],
\]
where $\alpha\in S_n$ is any permutation for which $\alpha(i) = 1$, $\alpha(j) = 2$, and $\alpha(k) = 3$.
\end{itemize}
That is, $[C_{ij}f]$ is the contraction of $[f]$ obtained by joining intervals $i$ and $j$, while $[C_{ijk}f]$ is the contraction obtained by joining intervals $i$ and~$j$, and then joining the result with~$k$.  Note that these contractions do not depend on the chosen permutation~$\alpha$, although they do depend on the chosen representative $f$ of the vertex~$[f]$.

\begin{proposition}Let $[f]$ be a vertex.  Then every contraction of $[f]$ has the form
\[
[C_{ij} u_i u_j f],\qquad [C_{ij} \sigma_i u_i u_j f],\qquad\text{or}\qquad [C_{ijk} u_i u_j u_k f].
\]
for some distinct $i,j,k\in\{1,\ldots,n\}$, where each $u_s \in \{\1,b_s,c_s,d_s\}$.
\end{proposition}
\begin{proof}This is similar to the proofs of Propositions~\ref{prop:ClassifySplittings} and Theorem~\ref{thm:double-splitting}.
\end{proof}

If $[f]$ is a vertex and $v$ is a contraction of $[f]$, we define the \newword{support} of $v$ (with respect to $f$) as follows:
\[
\supp(v) \;=\; \begin{cases}\{i,j\} & \text{if }v = [C_{ij}u_iu_jf]\text{ or } v = [C_{ij}\sigma_iu_i u_jf] \\ \{i,j,k\} & \text{if } v = [C_{ijk}u_iu_ju_kf].\end{cases}
\]
That is, the support of $v$ consists of those intervals which are joined together during the contraction.

\begin{lemma}\label{lemma:DisjointSimplex}Let\/ $[f]$ be a vertex, and let $v_1,\ldots,v_m$ be contractions of~$[f]$.  If the supports of $v_1,\ldots,v_m$ are disjoint, then $v_1,\ldots,v_m$ and $[f]$ all lie in an $m$-cube in~$\Knew$.
\end{lemma}

\begin{proof}We give a proof by example, from which the general procedure should be apparent.  Suppose $[f]$ is a vertex of rank~$9$, and suppose we are given three contractions of $f$:
\[
v_1 \;=\; [C_{27} u_2 u_7 f],\qquad v_2 \;=\; [C_{395}u_3u_9u_5f],\qquad\text{and}\qquad v_3 \;=\; [C_{41} \sigma_4 u_4 u_1f ],
\]
where each $u_i \in \{\1,b_i,c_i,d_i\}$.  The supports here are $\{2,7\}$, $\{3,5,9\}$, and $\{1,4\}$, respectively, so these three contractions have disjoint supports.

To construct a $3$-cube containing $[f]$, $v_1$, $v_2$, and $v_3$, we begin by choosing any permutation $\alpha\in S_9$ that agrees with the following table:
\[
\renewcommand{\arraystretch}{1.25}
\begin{array}{|c|c|c|c|c|c|c|c|}
\hline
x & 2 & 7 & 3 & 9 & 5 & 4 & 1 \\
\hline
\alpha(x) & 1 & 2 & 3 & 4 & 5 & 6 & 7 \\
\hline
\end{array}
\]
Let
\[
g \;=\; p_\alpha u_1 u_2 u_3 u_4 u_5 u_7 u_9f
\]
Then $[g] = [f]$, and
\[
v_1 \;=\; \bigl[x_1^{-1} g\bigr],\qquad v_2 \;=\; \bigl[x_3^{-1} x_3^{-1} g\bigr],\qquad\text{and}\qquad v_3 \;=\; \bigl[x_6^{-1} \sigma_6 g\bigr].
\]
Thus $[f]$, $v_1$, $v_2$, and $v_3$ are all contained in the $3$-cube
\[
\stope\bigl(\bigl[x_1^{-1}x_3^{-1}x_3^{-1}x_6^{-1}\sigma_6 g\bigr],\{\1,x\},\{1,x_1 x\},\{1,\sigma_1x\},\{\1\},\{\1\}\bigr)
\]
The same procedure works for any set of contractions with disjoint supports.
\end{proof}

\begin{lemma}\label{lemma:SupportTypes}Let\/ $[f]$ be a vertex in $\P$, let $v$ and $w$ be contractions of\/~$[f]$ whose corresponding vertices in $\dlink([f])$ are joined by an edge.  Then either one of the sets\/ $\supp_f(v)$ and\/ $\supp_f(w)$ is strictly contained in the other, or the two sets are disjoint.
\end{lemma}
\begin{proof}Since the vertices corresponding to $v$ and $w$ share an edge in $\dlink([f])$,  the vertices $v$, $w$, and $[f]$ must all lie on a common $2$-cell in $\Knew^{\leq\rank(f)}$, which must be either a triangle or a square.  If it is a square then $\supp_f(v)$ and $\supp_f(w)$ must be disjoint.  If it is a triangle then, assuming $\rank(v) \leq \rank(w)$, we have $v=[g]$, $w=[x_ig]$ and $[f] = [x_ix_ig]$ for some $g$ and~$i$, and it follows that $\supp_f(w)$ is strictly contained in~$\supp_f(v)$.
\end{proof}

\begin{lemma}\label{lemma:FlagComplex}Let $f\colon C(1)\to C(n)$ be a homeomorphism in $\V\G$.  Then the descending link\/ $\dlink([f])$ is a flag complex.
\end{lemma}
\begin{proof}Let $v_1,\ldots,v_r$ be contractions of $[f]$, and suppose that the corresponding vertices of $\dlink([f])$ are all connected by edges in the $1$-skeleton.  Let $\{v_1',\ldots,v_m'\}$ be the subset of $\{v_1,\ldots,v_r\}$ consisting of vertices with maximal support.
Then the supports of $v_1',\ldots,v_m'$ with respect to $f$ must be disjoint, so by Lemma~\ref{lemma:DisjointSimplex} $v_1',\ldots,v_m'$ and $[f]$ all lie in an $m$-cube in~$\Knew$.  This cube can be written as
\[
\stope(g,\{\1,u_1\},\ldots,\{\1,u_m\},\{\1\},\ldots,\{\1\}),
\]
where
\[
v_i' \;=\; \bigl[(u_1\oplus\cdots \oplus \1 \oplus\cdots \oplus u_m \oplus \1 \oplus \cdots\oplus \1)g\bigr]
\]
for each~$i$.

Now, if $v_k$ is not maximal, then $v_k$ will not be a vertex of this cube.  However, for each $v_i'$ there exists at most one $v_k$ so that $v_i' < v_k < [f]$ is an elementary simplex, since no two such $v_k$'s have disjoint supports.  In this case, there exists a $u_i' \in \{x,\sigma_1x\}$ so that
\[
v_k \;=\; \bigl[(u_1\oplus\cdots \oplus u_i' \oplus\cdots \oplus u_m \oplus \1 \oplus \cdots\oplus \1)g\bigr].
\]
Let $S_i = \{\1,u_i',u_i\}$ in this case, and let $S_i = \{1,u_i\}$ otherwise.  Then the polysimplex
\[
\stope(g,S_1,\ldots,S_m,\{\1\},\ldots,\{\1\}),
\]
contains all of the vertices $v_1,\ldots,v_r$ as well as $[f]$.
\end{proof}

We are now ready to analyze the connectivity of the descending links in $\Knew$.

\begin{proposition}Let $k\in\N$, and let $v$ be a vertex of $\Knew$ of rank at least~$\mbox{6k+2}$.  Then $\dlink(v)$ is $(k-1)$-connected.
\end{proposition}
\begin{proof}By Lemma \ref{lemma:FlagComplex}, the descending link $\dlink(v)$ is a flag complex.  We claim that $\dlink(v)$ is $(3k,3)$-grounded.  Let $f$ be a representative for~$v$, and let $w_1,\ldots,w_{3k+1}$ be the vertices $[C_{12}f],[C_{34}f],\ldots,[C_{6k+1,6k+2}]$.  Since the supports of the $w_i$'s are disjoint, by Lemma~\ref{lemma:DisjointSimplex} the corresponding vertices of $\dlink(v)$ form a \mbox{$3k$-simplex~$\Delta$}. Furthermore, if $w$ is any contraction of~$v$, then the support of $w$ is a set with at most three elements, which can intersect the supports of at most three different~$w_i$.  Then the vertex of $\dlink(v)$ corresponding to $w$ is connected to at least $(3k+1)-3$ vertices of~$\Delta$, which proves that $\dlink(v)$ is $(3k,3)$-grounded.  By Theorem~\ref{thm:Grounded}, we conclude that $\dlink(v)$ is $(k-1)$-connected.
\end{proof}

This concludes the proof of the Main Theorem.

\bigskip
\section*{Acknowledgments}

The authors would like to thank Collin Bleak, Kai-Uwe Bux, Elisabeth Fink,
Marco Marschler and Stefan Witzel for helpful conversations.
We would also like to thank the organizers of the 2013 LMS--EPSRC Durham Symposium on Geometric and Cohomological Group Theory,
as well as the University of Durham, the Universit\'e Paris Sud 11, the Universit\"{a}t Bielefeld
and the Universidade de Lisboa for their hospitality while this research was being
carried out. The first author would like to thank Bard College for the support he received during his junior sabbatical, and the second author gratefully acknowledges the Fondation Math\'ematique
Jacques Hadamard (ANR­-10-CAMP­-0151-02 - FMJH - Investissement d'Avenir)
and the Funda\c{c}\~ao para a Ci\^encia e a Tecnologia (FCT - PEst-OE/MAT/UI0143/2014)
for support received during the development of this work.

\bigskip
\bibliographystyle{plain}

\end{document}